\documentclass{amsart}
\usepackage{graphicx}
\usepackage{amsmath}
\usepackage{amsfonts}
\usepackage{amssymb}

\newtheorem{theorem}{Theorem}[section]

\theoremstyle{definition}

\theoremstyle{remark} \theoremstyle{remark}
\newtheorem{remark}[theorem]{Remark}
\newtheorem{example}[theorem]{Example}

\numberwithin{equation}{section}

\usepackage{color}

\begin{document}

\title[Paracontact metric manifolds]{Paracontact metric manifolds without a contact metric counterpart}

\author[V. Mart\'{i}n-Molina]{Ver\'onica  Mart\'{i}n-Molina}
 \address{Centro Universitario de la Defensa, Academia General Militar, Ctra. de Huesca s/n, 50090 Zaragoza, SPAIN\\
 AND I.U.M.A, Universidad de Zaragoza, SPAIN}
 \email{vmartin@unizar.es}

\thanks{The author is is partially supported by the PAI group
FQM-327 (Junta de Andaluc\'ia, Spain), the group Geometr\'ia E15 (Gobierno de Arag\'on, Spain),  the MINECO grant MTM2011-22621 and the ``Centro Universitario de la Defensa de Zaragoza'' grant ID2013-15.}

\begin{abstract}
We study non-paraSasakian paracontact metric $(\kappa,\mu)$-spaces with $\kappa=-1$ (equivalent to $h^2=0$ but $h\neq0$). These manifolds, which do not have a contact geometry counterpart, will be classified locally in terms of the rank of $h$. We will also give explicit examples of every possible constant rank of $h$.
\end{abstract}



\maketitle

\section{Introduction}

A remarkable class of paracontact metric manifolds $(M,\phi,\xi,\eta,g)$ is that of paracontact metric $(\kappa,\mu)$-spaces, which satisfy the nullity condition
\begin{equation}\label{kappamu}
R(X,Y)\xi=\kappa(\eta(Y)X-\eta(X)Y)+\mu(\eta(Y)hX-\eta(X) hY),
\end{equation}
for all $X,Y$ vector fields on $M$, where $\kappa$ and $\mu$ are constants and $h=\frac12 L_\xi \varphi$.

This definition, which may appear quite technical, arises from the deep and meaningful relationship between contact metric $(\kappa,\mu)$-spaces and paracontact geometry. More precisely, it was proved in \cite{mino-pacific} that any non-Sasakian contact metric $(\kappa,\mu)$-space accepts two paracontact metric $(\widetilde\kappa,\widetilde\mu)$-structures with the same contact form. On the other hand, under certain natural conditions, every non-paraSasakian paracontact $(\widetilde{\kappa},\widetilde{\mu})$-space carries a contact metric $(\kappa,\mu)$-structure compatible with the same contact form (\cite{CKM}).
The class of paracontact metric $(\kappa,\mu)$-spaces includes the paraSasakian ones (see \cite{kaneyuki} and \cite{zamkovoy}) and the ones satisfying $R(X,Y)\xi=0$ for all $X,Y$ (studied in \cite{zamkovoy-arxiv2}), among others.

There are some notable differences between a contact metric $(\kappa,\mu)$-space $(M,\phi,\xi,\eta,g)$ and a paracontact metric $(\widetilde\kappa,\widetilde \mu)$-space  $(\widetilde M,\widetilde\phi, \widetilde\xi, \widetilde\eta, \widetilde{g})$. First of all, while they satisfy  $h^2=(\kappa-1)\varphi^2$ and $\widetilde{h}^2=(\widetilde\kappa+1)\widetilde\varphi^2$, respectively, the first condition means that  $\kappa \leq 1$ but the second one does not give any type of restriction for $\widetilde\kappa$ because the metric of a paracontact metric manifold is not positive definite (see \cite{blair95} for the contact metric case and \cite{CKM} for the paracontact metric one).

Another difference is that, in the contact metric case, $\kappa=1$ (i.e.  $h^2=0$) is also  equivalent to the manifold being Sasakian (and thus $h=0$). However, there are paracontact metric $(\widetilde\kappa,\widetilde\mu)$-spaces with $\widetilde\kappa=-1$ (and thus $\widetilde{h}^2=0$) but $\widetilde h\neq0$. The first example of paracontact metric $(-1,2)$-space $(M^{2n+1},\widetilde\phi, \widetilde\xi, \widetilde\eta, \widetilde{g})$ with $\widetilde{h}\neq0$ was given in  \cite{mino-kodai} for $n=2$. Later, an example with arbitrary $n$ appeared in \cite{CKM} (constructed by deforming the contact metric structure defined on the unit tangent sphere bundle) and  a numerical example with $n=1$ was shown in \cite{murathan}. It is worth mentioning that all these spaces have $\widetilde\mu=2$ and $\text{rank} (\widetilde{h})=n$. Lastly, an example of $3$-dimensional paracontact metric $(-1,0)$-space with $\widetilde{h}\neq0$ appeared in \cite{CP}.

To our knowledge, no effort has been made to better understand the general behaviour of the tensor $h$ of a paracontact metric $(\kappa,\mu)$-space when $h^2=0$ but $h\neq0$, which we will address in Theorem \ref{th-h}.  We will study the form of the tensor $h$ in this remarkable situation and we will later construct explicit examples that illustrate all the possible constant values of the rank of $h$ (from $1$ to $n$) when $\mu=2$. Finally, we will discuss the situation when $\mu\neq2$ and show some paracontact metric $(-1,0)$-spaces with $h\neq0$ in Examples \ref{ex-mu0-h1}--\ref{ex-dim7-mu0-rank3}. These are the first examples of this type with $\mu\neq2$ and dimension greater than $3$.

\section{Preliminaries}\label{sec-preliminaries}

An \emph{almost paracontact structure} on a
$(2n+1)$-dimensional smooth manifold $M$ is given by a
$(1,1)$-tensor field $\varphi$, a vector field $\xi$ and a
$1$-form $\eta$ satisfying the following conditions \cite{kaneyuki}:
\begin{enumerate}
  \item[(i)] $\eta(\xi)=1$, \ $\varphi^2=I-\eta\otimes\xi$,
  \item[(ii)] the eigendistributions ${\mathcal D}^+$ and ${\mathcal D}^-$ of $\varphi$ corresponding to the eigenvalues $1$ and $-1$, respectively, have equal dimension $n$.
\end{enumerate}

As an immediate consequence, $\varphi\xi=0$, $\eta\circ\varphi=0$ and the tensor $\varphi$ has constant rank $2n$.
If an almost paracontact manifold is endowed with  a semi-Riemannian metric $ g$ such that
\[
g(\varphi X,\varphi Y)=- g(X,Y)+\eta(X)\eta(Y),
\]
for all  $X,Y$ on $M$, then $(M,\varphi,\xi,\eta, g)$ is called an \emph{almost paracontact metric manifold}. Note that such a
semi-Riemannian metric is necessarily of signature $(n+1,n)$ and the above condition (ii) of the definition of almost paracontact structures is
automatically satisfied. Moreover,  it follows easily that $\eta=g(\cdot,\xi)$ and
${g}(\cdot,\varphi\cdot)=-{g}(\varphi\cdot,\cdot)$. We can now define the \emph{fundamental $2$-form} of the almost paracontact metric manifold by
$\Phi(X,Y)={g}(X,\varphi Y)$. If $d\eta=\Phi$, then $\eta$ becomes a contact form (i.e. $\eta \wedge (d\eta)^n \neq0$) and $(M,\varphi,\xi,\eta, g)$ is said to be a \emph{paracontact
metric manifold}.

We can also define on a paracontact metric manifold the tensor field ${h}:=\frac{1}{2}L_{\xi}\varphi$, which is a symmetric operator with respect to ${g}$, anti-commutes with $\varphi$  and satisfies $h\xi=\text{tr} h=0$ and the  identity $\nabla\xi=-\varphi+\varphi h$ (\cite{zamkovoy}).  Moreover,  it  vanishes identically if and only if $\xi$ is a Killing vector
field, in which case $(M,\varphi,\xi,\eta, g)$ is called a \emph{K-paracontact manifold}.

An almost paracontact structure is said to be \emph{normal} if and only if the tensor $N_{\varphi}-2d\eta\otimes\xi$ vanishes identically, where $N_{\varphi}$ is the Nijenhuis tensor of $\varphi$: $N_{\varphi}(X,Y)=[\varphi,\varphi](X,Y)=\varphi^2 [X,Y]+[\varphi X,\varphi Y]-\varphi [\varphi X,Y]-\varphi [X,\varphi Y]$ (\cite{zamkovoy}). A normal paracontact metric manifold is said to be a \emph{paraSasakian manifold} and satisfies
\begin{equation}\label{parasasakian-r}
R(X,Y)\xi =-(\eta(Y)X-\eta(X)Y),
\end{equation}
for every $X,Y$ on $M$. Unlike in the contact metric case, the condition \eqref{parasasakian-r} does not imply that the manifold is paraSasakian, as will be seen in Examples \ref{ex-mu0-h1}--\ref{ex-dim7-mu0-rank3}.

It was also proved in \cite{zamkovoy} that an almost paracontact manifold is paraSasakian if and only if
\begin{equation}\label{parasasakian-phi}
(\nabla_{X}\varphi)Y=- g(X,Y)\xi+\eta(Y)X,
\end{equation}
so, in particular, every paraSasakian manifold is $K$-paracontact. The converse holds in dimension $3$ (\cite{calvaruso}) and for $(-1,\mu)$-spaces (which will be proved in Theorem \ref{th-sasakian}) but we can construct explicit examples that show that these two concepts are not equivalent in general.

\begin{example}\label{ex-k-contact-not-sasaki}
Let $\mathfrak{g}$ be the 5-dimensional Lie algebra with basis $\{\xi,X_1,Y_1,X_2,Y_2 \}$ such that the only non-vanishing Lie brackets are
\[
[X_1,Y_1]=2\xi, \quad [X_2,Y_2]=2\xi, \quad [X_1,X_2]=Y_1.
\]
If we denote by $G$ the Lie group whose Lie algebra is $\mathfrak{g}$, we can define a left-invariant paracontact metric structure on $G$ the following way:
\[
\varphi \xi=0,  \quad \varphi X_i=X_i,  \quad \varphi Y_i=-Y_i, \quad  \eta(\xi)=1,  \quad \eta(X_i)=\eta(Y_i)=0, \quad i=1,2,
\]
\[
g(\xi,\xi)=g(X_1,Y_1)=g(X_2,Y_2)=1,  \quad  g(\xi,X_i)=g(\xi,Y_i)=g(X_i,X_j)=g(Y_i,Y_j)=0, \quad i=1,2.
\]
A simple computation gives that $h=0$, so the manifold is $K$-paracontact.

However, although $R(X,\xi)\xi=-X$ for all vector field $X$ orthogonal to $\xi$, straightforward computations give that $R(X_1,X_2)\xi=-2Y_1\neq0$, so the manifold does not satisfy \eqref{parasasakian-r}. In particular, it is not paraSasakian.
\end{example}


Finally, we recall that a notion of  ${\mathcal D}_c$-homothetic deformation can be introduced in paracontact metric geometry
(\cite{zamkovoy}). Indeed, given a non-zero constant $c$, a \emph{${\mathcal D}_{c}$-homothetic deformation} on a paracontact metric manifold $(M,\varphi,\xi,\eta,g)$ is the following change of the structure tensors:
\begin{equation}\label{deformation}
\varphi':=\varphi, \quad \xi':=\frac{1}{c}\xi, \quad \eta':=c \eta, \quad g':=c g + c(c-1)\eta\otimes\eta.
\end{equation}
Then $(\varphi',\xi',\eta',g')$ is again a paracontact metric structure on $M$. Moreover, $K$-paracontact and paraSasakian structures are also preserved under ${\mathcal D}_c$-homothetic deformations.

Although $D_c$-homothetic deformations destroy curvature conditions like $R(X,Y)\xi=0$, a $D_c$-homothetically deformed paracontact metric $(\kappa,\mu)$-space is another paracontact metric $(\kappa',\mu')$-space with
\begin{equation}\label{kappamu-deformed}
\kappa'=\frac{\kappa+1-c^2}{c^2}, \quad \mu'=\frac{\mu-2+2c}{c}.
\end{equation}

\section{Paracontact metric $(-1,\mu)$-spaces}

We now present our main results.

\begin{theorem}\label{th-sasakian}
Let $M$ be a paracontact metric manifold such that $R(X,Y)\xi=-(\eta(Y)X-\eta(X)Y)$, for all $X,Y$ on $M$. Then $M$ is paraSasakian if and only if $\xi$ is a Killing vector field.
\end{theorem}
\begin{proof}
If $M$ is paraSasakian, it is in particular a $K$-paracontact manifold, hence $\xi$ is a Killing vector field.

Conversely, if $\xi$ is a Killing vector field, then $h=0$ and $\nabla_X \xi=-\varphi X$ holds for every $X$ on $M$ (\cite{zamkovoy}). We also know from \cite[p.259]{O} that
\[
R(\xi,X)Y=-\nabla_X\nabla_Y \xi+\nabla_{\nabla_X Y} \xi,
\]
thus $R(\xi,X)Y=(\nabla_X \varphi)Y$, for all $X,Y$ on $M$.

Therefore, it follows from the previous formula and from $R(X,Y)\xi=-(\eta(Y)X-\eta(X)Y)$  that
\[
g((\nabla_X \varphi)Y,Z)=g(R(\xi,X)Y,Z)=g(R(Y,Z)\xi,X)=g(-g(X,Y)\xi+\eta(Y)X,Z),
\]
hence equation \eqref{parasasakian-phi} holds and the manifold is paraSasakian.
\end{proof}

\begin{theorem}\label{th-h}
Let $M$ be a $(2n+1)$-dimensional paracontact metric $(-1,\mu)$-space. Then we have one of the following possibilities:
\begin{enumerate}
\item
$h=0$ and $M$ is paraSasakian

\item
$h\neq 0$ and $\text{rank} (h_p)\in \{1,\ldots,n \}$ at every $p  \in M$ where $h_p \neq 0$. Moreover, there exists a basis $\{ \xi, X_1,Y_1,\ldots,X_n,Y_n \}$ of $T_p(M)$ such that
\begin{itemize}
\item The only non-vanishing components of $g$ are
\[
g_p(\xi,\xi)=1, \quad g_p(X_i,Y_i)=\pm 1,
\]
\item and
\begin{equation*}
h_{| \langle X_i,Y_i \rangle}=
\begin{pmatrix}
0 & 0\\
1 & 0
\end{pmatrix}
\quad
\text{ or }
\quad
h_{| \langle X_i,Y_i \rangle}=
\begin{pmatrix}
0 & 0\\
0 & 0
\end{pmatrix},
\end{equation*}
where obviously there are exactly $\text{rank} (h_p)$ submatrices of the first type.
\end{itemize}
\end{enumerate}

If $n=1$, such a basis $\{ \xi, X_1,Y_1\}$ also satisfies that
\[
\varphi X_1 =\pm  X_1, \quad \varphi Y_1 =\mp  Y_1,
\]
and the tensor $h$ can be written as
\begin{equation}\label{h-nonzero-dim3}
h_{| \langle\xi, X_i,Y_i \rangle}=
\begin{pmatrix}
0 & 0 & 0\\
0 & 0 & 0\\
0 & 1 & 0
\end{pmatrix}.
\end{equation}

\end{theorem}
\begin{proof}
We know from Lemma 3.2 of \cite{CKM} that $h^2=0$. We have now two possibilities.

If $h=0$, then $R(X,Y)\xi=-(\eta(Y)X-\eta(X)Y)$, for all $X,Y$ on $M$ and $\xi$ is a Killing vector field. Therefore, it follows from Theorem~\ref{th-sasakian} that the manifold is paraSasakian.

Let us now suppose that $h\neq0$. Since $h$ is self-adjoint and $\text{Ker} (\eta)$ is $h$-invariant, we have from \cite[p.260]{O} that, at each point $p \in M$, $\text{Ker} (\eta_p)=V_1 \oplus \cdots \oplus V_l$ (for some $1\leq l\leq 2n$), where $V_k$ are mutually orthogonal subspaces that are $h$-invariant and on which $h_{|V_k}$ has a matrix of either type:
\begin{equation}\label{type1}
\left(
\begin{array}{ccccc}
\lambda &          &         &          &\\
1       & \lambda  &         &\mathbf{0}&\\
        &  1       & \lambda &          &\\
        &          & \ddots  &\ddots    &\\
        &\mathbf{0}&         &1         &\lambda
\end{array}
\right)
\end{equation}
relative to a basis $X_1,\ldots, X_r$ of $V_k$, $r\geq 1$, such that the only non-zero products are $g_p(X_i,X_j)=\pm 1$ if $i+j=r+1$, or of type
\begin{equation}\label{type2}
\left(
\begin{array}{ccccccccccc}
a  & b &    &   &    &    &       &    &    &    & \\
-b & a &    &   &    &    &       &    &  \mathbf{0}  &    & \\
1  & 0 & a  & b &    &    &       &    &    &    & \\
0  & 1 & -b & a &    &    &       &    &    &    & \\
   &   & 1  & 0 & a  & b  &       &    &    &    & \\
   &   & 0  & 1 & -b & a  &       &    &    &    & \\
   \\
   &   &    &   &    &\ddots&     &\ddots&  &    & \\
   &   &  \mathbf{0}  &   &    &    &       &    &    &    & \\
   &   &    &   &    &    &       & 1  & 0  & a  & b\\
   &   &    &   &    &    &       & 0  & 1  & -b & a
\end{array}
\right)
(b \neq 0)
\end{equation}
relative to a basis $X_1,Y_1,\ldots,X_m,Y_m$ of $V_k$, such that the only non-zero products are $g_p(X_i,X_j)=1=-g_p(Y_i,Y_j)$ if $i+j=m+1$.

We will first see that the second case is not possible. Indeed, if there existed a subspace $V_k$ such that $h_{|V_k}$ had a matrix of  type \eqref{type2}, then $h_p^2 X_1=(a^2-b^2)X_1-2ab Y_1+2a X_2-2b Y_2+X_3=0$, which cannot happen for any value of $m$  because $b\neq0$.

If there exists a subspace $V_k$ such that $h_{|V_k}$ has a matrix of  type \eqref{type1}, then $h_p^2 X_1=\lambda^2 X_1+2\lambda X_2+X_3=0$, which is only possible if $\lambda=0$ and $\dim V_k =r\leq2$. Let us distinguish between both subcases:
\begin{enumerate}
\item\label{Type-Ib} If $\dim V_i=2$, then $V_i=\langle X_i,Y_i \rangle$, the only non-zero product is $g_p(X_i,Y_i)=\pm 1$ and
\[
h_{| \langle X_i,Y_i \rangle }=
\begin{pmatrix}
0 & 0\\
1 & 0
\end{pmatrix}.
\]

\item\label{Type-Ia} If $\dim V_i=1$, then $V_i=\langle X_i \rangle$ and $h_pX_i=0$, with $X_i$ a vector satisfying $g_p(X_i,X_i)=\pm1$. In fact, since $\text{Ker} (\eta_p)$ is of signature $(n,n)$ and the subspaces $V_i$ of the subcase~\eqref{Type-Ib} are of dimension two and signature $(1,1)$, then there is an even number of subspaces of this type, half satisfying  $g_p(X_i,X_i)=1$ and half satisfying $g_p(X_i,X_i)=-1$. Taking one of each type, for example $V_1=\langle X_1 \rangle$ and $V_2=\langle X_2 \rangle$ with $g_p(X_1,X_1)=1=-g_p(X_2,X_2)$, a simple change of basis like $\widetilde{X_1}=\frac{1}{\sqrt2}(X_1+X_2)$ and $\widetilde{Y_1}=\frac{1}{\sqrt2}(X_1-X_2)$   would give us a basis such that the only non-vanishing component of the metric is $g_p(\widetilde{X}_1,\widetilde{Y}_1)=1$.
\end{enumerate}

Finally, since $h|_{V_i}=0$ in the second subcase, the rank of $h_p$ depends on the number of submatrixes of the first type, which have rank $1$. Therefore, $\text{rank}(h_p) \in \{1,\ldots,n \}$ at every point where $h_p \neq 0$.

In dimension $3$, the conditions over $\varphi$ are the only ones that remain to be proved. First of all, $\varphi \xi=0$ on every paracontact metric structure. Hence $\varphi X_1=aX_1+bY_1$, for some constants $a,b$.

It follows from $g_p(X_1,\varphi X_1)=d\eta_p(X_1,X_1)=0$ that $b=0$, so $\varphi X_1=aX_1$. On the other hand, $\varphi Y_1= \varphi h X_1=-h \varphi X_1=-a h X_1=-aY_1$, hence $\varphi Y_1=-aY_1$. We can then compute:
\begin{align*}
g_p(\varphi X_1,\varphi Y_1)&=-g_p(X_1,Y_1)=\mp 1, \\
g_p(\varphi X_1,\varphi Y_1)&=g_p(a X_1,-aY_1)=-a^2 g_p(X_1,Y_1)=\mp a^2,
\end{align*}
so $a^2=1$, and thus
\[
\varphi X_1 =\pm  X_1, \quad \varphi Y_1 =\mp  Y_1,
\]
which ends the proof.
\end{proof}

\begin{remark}
A paracontact metric manifold satisfying
 \begin{equation} \label{r-0}
R(X,Y)\xi=-(\eta(Y)X-\eta(X)Y),
\end{equation}
for all $X,Y$ on $M$, can be either a $(-1,\mu)$-space with $h=0$ (thus  paraSasakian by the previous Theorem and $\mu$ is undetermined) or a paracontact metric $(-1,0)$-space with $h\neq0$ (not $K$-paracontact or paraSasakian).
This last case is possible because \eqref{r-0} does not imply $h=0$, as can be seen in Examples \ref{ex-mu0-h1}--\ref{ex-dim7-mu0-rank3}.
\end{remark}


Let us now see examples of all the possible constant ranks of $h$ that appear in Theorem \ref{th-h}. First of all, if $h=0$,  the standard examples of paraSasakian manifolds are the hyperboloids
\begin{equation*}
\mathbb{H}^{2n+1}_{n+1}(1)=\left\{(x_{0},y_{0},\ldots,x_{n},y_{n})\in\mathbb{R}^{2n+2} \ | \
x_{0}^{2}+\ldots+x_{n}^{2}-y_{0}^{2}-\ldots-y_{n}^{2}=1\right\}
\end{equation*}
and the hyperbolic Heisenberg group ${\mathcal H}^{2n+1}=\mathbb{R}^{2n}\times\mathbb{R}$ with the structures defined in \cite{ivanov}. Other examples of ($\eta$-Einstein) paraSasakian manifolds can be obtained from contact $(\kappa,\mu)$-spaces with $|1-\frac{\mu}{2}|<\sqrt{1-\kappa}$, as seen in Theorem~3.4 of \cite{nuestro-mino}. In particular, it was shown that the tangent sphere bundle $T_1N$ of any space form $N(c)$ with $c<0$ admits a canonical $\eta$-Einstein paraSasakian structure.

We can also construct paraSasakian examples by defining a paracontact metric structure on a Lie group:
\begin{example}[Canonical paraSasakian structure on the Heisenberg algebra]
Let $\mathfrak{g}$ be the $(2n+1)$-dimensional Lie algebra with basis $\{\xi,X_1,\ldots,X_{2n} \}$ such that the only non-vanishing Lie brackets are
\[
[X_{2i-1},X_{2i}]=2\xi, \quad  i=1,\ldots,n.
\]
If we denote by $G$ the Lie group whose Lie algebra is $\mathfrak{g}$, we can define a left-invariant paracontact metric structure on $G$ the following way:
\[
\varphi \xi=0,  \quad \varphi X_{2i-1}=X_{2i}, \quad \varphi X_{2i}=X_{2i-1},  \quad i=1,\ldots,n,
\]
\[
 \eta(\xi)=1,  \quad \eta(X_i)=0, \quad  i=1,\ldots,2n,
\]
the only non-vanishing components of the metric are
\[
g(\xi,\xi)=g(X_{2i},X_{2i})=1, \quad g(X_{2i-1},X_{2i-1})=-1, \quad i=1,\ldots,n.
\]
A straightforward computation gives that $h=0$. Moreover, using properties of paracontact metric manifolds and Koszul's formula, we obtain that
\[
\nabla_{X_{2i}} \xi=-X_{2i-1}, \quad \nabla_{X_{2i-1}} \xi=-X_{2i}, \quad i=1,\ldots,n,
\]
\[
\nabla_{X_{2i}} X_{2i}=\nabla_{X_{2i-1}} X_{2i-1}=0, \quad \nabla_{X_{2i-1}} X_{2j}=-\nabla_{X_{2i}} X_{2j-1}=\delta_{ij} \xi, \quad i=1,\ldots,n.
\]
Therefore,
\begin{align*}
R(X_i,\xi)\xi &=-X_i, \quad i=1,\ldots,2n,\\
R(X_i,X_j)\xi &=0,  \quad   i,j=1,\ldots,2n.
\end{align*}
We conclude that $R(X,Y)\xi=-(\eta(Y)X-\eta(X)Y)$ for all $X,Y$ on $M$, thus the manifold is paraSasakian because of Theorem~\ref{th-sasakian}. Alternatively, we could check the normality condition by proving that
$[\varphi, \varphi](X,Y)-2 d \eta (X,Y) \xi=0$ for all vector fields $X,Y$ on $M$.
\end{example}

If we apply a $D_c$-homothetic deformation to any of the previous paraSasakian examples, we obtain again paraSasakian manifolds (\cite{zamkovoy}).



\medskip

Let us now construct some examples of paracontact metric $(-1,\mu)$-spaces with $h\neq0$. We will begin by adapting some paracontact metric $(-1,2)$-spaces with $h\neq0$ that appear in the literature and will afterwards construct explicitly $(2n+1)$-dimensional paracontact metric $(-1,2)$-spaces with $h\neq0$ such that rank of $h$ attains all values in $\{1,\ldots,n\}$. Lastly, the case $\mu\neq2$ will be discussed.

\begin{example}[$(2n+1)$-dimensional paracontact metric $(-1,2)$-space, $\text{rank}(h)=n$] \label{ex-tangent-mu2}
We will begin with the examples that were presented in \cite{CKM}. Let us take a flat Riemannian manifold $M$ and construct on it the tangent sphere bundle $T_1M$ with its standard contact metric structure $(\varphi,\xi,\eta,g)$, which satisfies $R(X,Y)\xi=0$ for every $X,Y$ on $T_1M$ (see \cite{blair95}).  Then we can define a new structure by taking
\begin{equation*}
\varphi_2=h, \quad g_2=d\eta(\cdot,h\cdot)+\eta \otimes \eta,
\end{equation*}
which is a paracontact metric $(-1,2)$-space (Theorem 3.4 of \cite{CKM}). We will now see the form that $h_2$ has on these examples.

Let us first take a $\varphi$-basis $\{ \xi,X_1,\varphi X_1,\ldots,X_n,\varphi X_n \}$ of the contact metric $(\kappa,\mu)$-space such that $hX_i=X_i$ and $h \varphi X_i=-\varphi X_i$ (which exists because of \cite{blair95}). Then  $\{ \xi,\widetilde{X}_1, \widetilde{Y}_1,\ldots,\widetilde{X}_n, \widetilde{Y}_n \}$, where $\widetilde{X}_i:=\frac{1}{\sqrt2}X_i$, $\widetilde{Y}_i:=\sqrt2 \varphi X_i$, is a basis for which the only non-vanishing components of the metric $g_2$ are
\[
g_2(\xi,\xi)=1, \quad g_2(\widetilde{X}_i,\widetilde{Y}_i)= 1,
\]
the tensor $\varphi_2$ satisfies
\[
\varphi_2 \widetilde{X}_i =  \widetilde{X}_i, \quad \varphi_2 \widetilde{Y}_i =-  \widetilde{Y}_i,
\]
and
\[
{h_2}_{| \langle \widetilde{X}_i,\widetilde{Y}_i \rangle }=
\begin{pmatrix}
0 & 0\\
1 & 0
\end{pmatrix},
\]
for all $i=1,\ldots,n$.

Therefore,
\[
h_2=
\begin{pmatrix}
0  &   &     &   &        &    &\\
   & 0 & 0   &   &        &    & \\
   & 1 & 0   &   &        &    & \\
   &   &     &   &\ddots^{\scriptstyle{(n)}}  &    & \\
   &   &     &   &        & 0  & 0\\
   &   &     &   &        & 1  & 0
\end{pmatrix},
\]
so $\text{rank}(h_2)=n$.
\end{example}

In dimension 3, we also have Example 6.2 from \cite{murathan}, which is a $3$-dimensional paracontact metric $(-1,2)$-space with $h\neq0$ satisfying \eqref{h-nonzero-dim3}.

\medskip

On the other hand, we can give examples using left-invariant paracontact metric structures on Lie groups of $(2n+1)$-dimensional paracontact metric $(-1,2)$-spaces with $h\neq0$ and $\text{rank}(h)=m \in \{1,\ldots,n \}$.

\begin{example}[$(2n+1)$-dimensional paracontact metric $(-1,2)$-space with $\text{rank}(h)=m \in \{1,\ldots,n \}$] \label{ex-mu2-hm-n}

Let $\mathfrak{g}$ be the $(2n+1)$-dimensional Lie algebra with basis $\{\xi,X_1,Y_1,\ldots,X_{n},Y_{n} \}$ such that the only non-vanishing components are
\begin{align*}
[\xi,X_i]=Y_i, \quad [X_i,Y_j]&=
\left\{
\begin{array}{l}
\delta_{ij} (2\xi +\sqrt2 (1+\delta_{im})Y_m)+(1-\delta_{ij})\sqrt2 (\delta_{im}Y_j+\delta_{jm} Y_i),  \quad i,j=1,\ldots,m,\\
\delta_{ij}(2\xi+\sqrt2 Y_i),                                                                          \quad i,j=m+1, \ldots,n, \\
\sqrt2 Y_i                                                                                             \quad i=1, \ldots,m, \; j=m+1, \ldots,n.
\end{array}
\right.
\end{align*}

If we denote by $G$ the Lie group whose Lie algebra is $\mathfrak{g}$, we can define a left-invariant paracontact metric structure on $G$ the following way:
\[
\varphi \xi=0,  \quad \varphi X_i=X_i, \quad  \varphi Y_i=-Y_i,  \quad i=1, \ldots,n,
\]
\[
\eta(\xi)=1,  \quad \eta(X_i)=\eta(Y_i)=0, \quad i=1,\ldots,n.
 \]
The only non-vanishing components of the metric are
\[
g(\xi,\xi)=g(X_i,Y_i)=1, \quad i=1,\ldots,n.
\]
A straightforward computation gives that $h X_i=Y_i$ if $i=1,\ldots,m$, $h X_i=0$ if $i=m+1,\ldots,n$ and $h Y_j=0$ if  $j=1, \ldots,n$, so $h^2=0$ and $\text{rank}(h)=m$.

Moreover, by basic paracontact metric properties we obtain that $\nabla_{X_i} \xi=-X_i-Y_i$ if $i=1,\ldots,m$, $\nabla_{X_i} \xi=-X_i$ if $i=m+1,\ldots,n$ and $\nabla_{Y_i} \xi=Y_i$ if $i=1,\ldots,n$, from which we deduce that
\begin{align*}
R(X_i,\xi)\xi &=-X_i+2Y_i=-X_i+2hX_i, \quad i=1,\ldots,m,\\
R(X_i,\xi)\xi &=-X_i, \quad i=m+1,\ldots,n,\\
R(Y_i,\xi)\xi &=-Y_i, \quad i=1,\ldots,n.
\end{align*}
Therefore, $R(X,\xi)\xi=-X+2hX$ for every $X$ orthogonal to $\xi$.

Using Koszul's formula, we can compute $\nabla_{X_i} Y_j$ and $\nabla_{Y_i} Y_j$ for any $i,j$:
\[
\begin{aligned}
\nabla_{X_i} Y_j&=
\left\{
\begin{array}{ll}
\delta_{ij} (\xi+\sqrt2 (1+\delta_{im} Y_m))+(1-\delta_{ij})\sqrt2 (\delta_{im} Y_j+\delta_{jm} Y_i), & i,j=1,\ldots,m, \\
\sqrt2 Y_i,                   & i=1,\ldots,m,\; j=m+1,\ldots,n, \\
0,                            & i=m+1,\ldots,n,\; j=1,\ldots,m,  \\
\delta_{ij} (\xi+\sqrt2 Y_i), & i,j=m+1,\ldots,n,
\end{array}
\right.
\\
\nabla_{Y_i} Y_j&=0,                            \quad i,j=1,\ldots,n.
\end{aligned}
\]
It follows that $R(X_i,X_j)\xi=R(X_i,Y_j)\xi=R(Y_i,Y_j)\xi=0$ for any $i,j=1,\ldots,n$, which is enough to conclude that the manifold is a $(-1,2)$-space.
\end{example}

If we take $m=0$ in the previous example, we will obtain a $(2n+1)$-dimensional paracontact metric $(-1,2)$-space such that $\text{rank}(h)=n$, as in Example \ref{ex-tangent-mu2}. However, the construction of our manifold has been considerably different.

We can also change the Lie algebra and paracontact structure of Example \ref{ex-mu2-hm-n} to obtain other possibilities. For example, in the particular case $m=1$, we have the following one.

\begin{example}[$(2n+1)$-dimensional paracontact metric $(-1,2)$-space with  $\text{rank}(h)=1$] \label{ex-mu2-h1}
Let $\mathfrak{g}$ be the $(2n+1)$-dimensional Lie algebra with basis $\{\xi,X_1,Y_1,\ldots,X_n,Y_n \}$ such that the only non-vanishing components are
\begin{gather*}
[\xi,X_1]=Y_1, \quad [X_1,Y_1]=2 \xi,  \\
[X_i,Y_i]=2(\xi+X_i), \quad [X_1,X_i]=Y_1, \quad i=2,\ldots,n.
\end{gather*}

If we denote by $G$ the Lie group whose Lie algebra is $\mathfrak{g}$, we can define a left-invariant paracontact metric structure on $G$ the following way:
\[
\varphi \xi=0,  \quad \varphi X_1=X_1, \quad  \varphi Y_1=-Y_1,  \quad\varphi X_i=-X_i, \quad  \varphi Y_i=Y_i, \quad i=2, \ldots,n,
\]
\[
\eta(\xi)=1,  \quad \eta(X_i)=\eta(Y_i)=0, \quad i=1,\ldots,n.
 \]
The only non-vanishing components of the metric are
\[
g(\xi,\xi)=g(X_1,Y_1)=1, \quad g(X_i,Y_i)=-1, \quad i=2,\ldots,n.
\]
A straightforward computation gives that $h X_1=Y_1$, $h Y_1=0$ and $h X_i=hY_i=0$, $i=2, \ldots,n$, so $h^2=0$ and $\text{rank}(h)=1$.

Moreover, by basic paracontact metric properties and Koszul's formula we obtain that
\begin{gather*}
\nabla_\xi X_1=-X_1,  \quad \nabla_\xi Y_1=Y_1, \quad \nabla_\xi X_i=X_i, \quad  \nabla_\xi Y_i=-Y_i, \quad i=2,\ldots,n,\\
\nabla_{Y_i} Y_j=2\delta_{ij}(1-\delta_{i1})Y_i, \quad  \nabla_{X_i} Y_j=\delta_{ij} \xi, \quad \nabla_{Y_i} X_j=\delta_{ij} (-\xi+2(1-\delta_{i1})X_i), \quad i,j=1,\ldots,n,\\
\nabla_{X_i} X_1=0, \quad \nabla_{X_1} X_j=Y_1, \quad i=2,\ldots,n,
\end{gather*}
which is enough to prove that
\begin{align*}
R(X_1,\xi)\xi &=-X_1+2Y_1=-X_1+2hX_1,\\
R(X_i,\xi)\xi &=-X_i, \quad i=2,\ldots,n,\\
R(Y_i,\xi)\xi &=-Y_i, \quad i=1,\ldots,n,\\
R(X_i,X_j)\xi &=R(X_i,Y_j)\xi=R(Y_i,Y_j)\xi=0, \quad i,j=1,\ldots,n.
\end{align*}
Therefore, the manifold is also a $(-1,2)$-space.
\end{example}

Note that all the examples constructed until now and most of the ones appearing in the literature of paracontact metric $(-1,\mu)$-spaces with $h^2=0$ but $h\neq0$  have $\mu=2$. Why is this particular value so important? First of all, given a non-Sasakian contact metric $(\kappa,\mu)$-space, we can define two canonical paracontact metric structures on $M$, $(\widetilde{\varphi}_1,\xi,\eta,\widetilde{g}_1)$ and $(\widetilde{\varphi}_2,\xi,\eta,\widetilde{g}_2)$, in the following way (\cite{mino-kodai} and \cite{mino-pacific}):
\begin{align*}
\widetilde\varphi_{1}&:=\frac{1}{\sqrt{1-\kappa}}\varphi h, &\quad
\widetilde{g}_{1}&:=\frac{1}{\sqrt{1-\kappa}}g(\cdot, h \cdot)+\eta\otimes\eta,  \\
\widetilde\varphi_{2}&:=\frac{1}{\sqrt{1-\kappa}}h, &\quad
\widetilde{g}_{2}&:=\frac{1}{\sqrt{1-\kappa}}g(\cdot,\varphi h \cdot)+\eta\otimes\eta.
\end{align*}
Moreover, these new structures are paracontact metric $(\widetilde{\kappa}_i,\widetilde{\mu}_i)$-spaces with
\begin{align*}
\widetilde\kappa_1&=\left(1-\frac{\mu}{2}\right)^2-1, & \widetilde\mu_1&=2\left(1-\sqrt{1-\kappa}\right), & \widetilde{h}_1&=-\frac{1-\frac{\mu}2}{\sqrt{1-\kappa}} h,\\
\widetilde\kappa_2&=\kappa-2+\left(1-\frac{\mu}{2}\right)^2, & \widetilde\mu_2&=2, & \widetilde{h}_2&=\frac{1-\frac{\mu}2}{\sqrt{1-\kappa}} \varphi h +\sqrt{1-\kappa} \varphi.
\end{align*}
Therefore, $\widetilde\kappa_1=-1$ if and only if $\mu=2$, which is equivalent to $(\widetilde{\varphi}_1,\xi,\eta,\widetilde{g}_1)$ being paraSasakian (\cite[Th. 5.9]{mino-kodai} or Theorem \ref{th-h}).  On the other hand, $\widetilde\kappa_2=-1$ if and only if $\sqrt{1-\kappa}=1-\frac{\mu}2$, from which it follows that $h_2=\varphi h+(1-\frac{\mu}{2})\varphi$. This means that the only possibility of constructing a paracontact metric $(\widetilde\kappa,\widetilde\mu)$-space with $\widetilde{h}^2=0$ and $\widetilde{h}\neq0$ using these structures is to take the second one, which has $\widetilde{\mu}=2$.

Another reason for which the value $\mu=2$ is special comes from the fact that applying a $D_c$-homothetic deformation to a paracontact metric $(-1,2)$-space gives us another paracontact metric $(-1,2)$-space for any real $c\neq0$.

What happens when $\mu\neq2$? Are there any examples of paracontact metric $(-1,\mu)$-spaces with $h\neq0$ and $\mu\neq2$? The answer is yes, as was shown in Example~4.6 of \cite{CP} for the $3$-dimensional case. We will provide proof in dimensions greater than $3$ in Examples \ref{ex-mu0-h1}--\ref{ex-dim7-mu0-rank3}. Before constructing them, note that, given a  $(-1,\mu)$-space  with $\mu\neq2$, a  $\mathcal{D}_c$-homothetic deformation with $c=1-\frac{\mu}{2}\neq 0$ will give us a paracontact metric $(-1,0)$-space thanks to \eqref{kappamu-deformed}. Conversely, given a paracontact metric $(-1,0)$-space, if we $\mathcal{D}_c$-homothetically deform it with $c=\frac2{2-\mu}\neq 0$ for some $\mu\neq2$,  we will obtain a paracontact metric  $(-1,\mu)$-space with $\mu\neq2$.

Lastly, the rank of $h$ and $h'$ coincides because $h'=\frac{1}{c}h$ (\cite{CKM}) and a simple change of basis will give us another one satisfying all the properties of Theorem \ref{th-h}. Therefore, it makes sense to concentrate on the $\mu=0$ case, which is also special because the paracontact metric $(-1,0)$-spaces with $h\neq0$ satisfy \eqref{parasasakian-r} but are not paraSasakian manifolds.

We can give examples using left-invariant paracontact metric structures on Lie groups of $(2n+1)$-dimensional paracontact metric $(-1,0)$-spaces with $h\neq0$. We recall that $\text{rank} (h_p) \in \{ 1,\ldots,n \}$ (Theorem \ref{th-h}) and we will first see the case $\text{rank} (h)=1$.

\begin{example} [$(2n+1)$-dimensional paracontact metric $(-1,0)$-space with  $\text{rank}(h)=1$] \label{ex-mu0-h1}
Let $\mathfrak{g}$ be the $(2n+1)$-dimensional Lie algebra with basis $\{\xi,X_1,Y_1,\ldots,X_n,Y_n \}$ such that the only non-vanishing components are
\begin{gather*}
[\xi,X_1]=X_1+Y_1, \quad [\xi,Y_1]=-Y_1, \quad [X_1,Y_1]=2 \xi, \\
[X_i,Y_i]=2(\xi+Y_i), \quad [X_1,Y_i]=X_1+Y_1, \quad [Y_1,Y_i]=-Y_1, \quad i=2,\ldots,n.
\end{gather*}

If we denote by $G$ the Lie group whose Lie algebra is $\mathfrak{g}$, we can define a left-invariant paracontact metric structure on $G$ the following way:
\[
\varphi \xi=0,  \quad \varphi X_1=X_1, \quad  \varphi Y_1=-Y_1,  \quad\varphi X_i=-X_i, \quad  \varphi Y_i=Y_i, \quad i=2, \ldots,n,
\]
\[
\eta(\xi)=1,  \quad \eta(X_i)=\eta(Y_i)=0, \quad i=1,\ldots,n.
 \]
The only non-vanishing components of the metric are
\[
g(\xi,\xi)=g(X_1,Y_1)=1, \quad g(X_i,Y_i)=-1, \quad i=2,\ldots,n.
\]
A straightforward computation gives that $h X_1=Y_1$, $h Y_1=0$ and $h X_i=hY_i=0$, $i=2, \ldots,n$, so $h^2=0$ and $\text{rank}(h)=1$.

Moreover, by basic paracontact metric properties and Koszul's formula we obtain that
\begin{gather*}
\nabla_\xi X_1=0,  \quad \nabla_\xi Y_1=0, \quad \nabla_\xi X_i=X_i, \quad  \nabla_\xi Y_i=-Y_i, \quad i=2,\ldots,n,\\
\nabla_{X_i} Y_1=\delta_{i1}\xi, \quad \nabla_{X_i} Y_j=\delta_{ij} (\xi+2Y_i), \quad \nabla_{Y_1} X_1=-\xi, \quad \nabla_{Y_i} X_j=-\delta_{ij} \xi, \quad i,j=2,\ldots,n,\\
\nabla_{X_1} X_j=0,  \nabla_{Y_1} Y_1=\nabla_{Y_1} Y_j=0, \quad \nabla_{Y_j} Y_1=Y_1, \quad i=2,\ldots,n,
\end{gather*}
and thus
\begin{align*}
R(X_i,\xi)\xi &=-X_i, \quad i=1,\ldots,n,\\
R(Y_i,\xi)\xi &=-Y_i, \quad i=1,\ldots,n,\\
R(X_i,X_j)\xi &=R(X_i,Y_j)\xi=R(Y_i,Y_j)\xi=0, \quad i,j=1,\ldots,n.
\end{align*}
Therefore, the manifold is also a $(-1,0)$-space.
\end{example}

To our knowledge, the previous example is the first paracontact metric $(-1,\mu)$-space with $h^2=0$, $h\neq0$ and  $\mu\neq2$ that has been constructed in dimensions greater than $3$. For dimension $3$, Example~4.6 of \cite{CP} was already known.

Let us now see some other possibilities. In dimension $3$, $n=1$, so the previous example gives the only possible value of the rank of $h$.  In dimension $5$, we can construct an example with  $\text{rank}(h)=2$ in the following way:

\begin{example}[$5$-dimensional paracontact metric $(-1,0)$-space with  $\text{rank}(h)=2$] \label{ex-dim5-mu0}
Let us take the $5$-dimensional Lie algebra $\mathfrak{g}$ of basis $\{\xi,X_1,Y_1,X_2,Y_2 \}$ and whose Lie brackets are
\begin{align*}
[\xi,X_1]&=Y_1+X_1+X_2,                 & [\xi,Y_1]&=-Y_1+Y_2,\\
[\xi,X_2]&=Y_2+X_1+X_2,                 & [\xi,Y_2]&=Y_1-Y_2, \\
[X_1,Y_1]&=2\xi+X_2+\frac32 Y_2,        & [X_2,Y_2]&=-2\xi+X_1-\frac12 Y_1, \\
[X_1,X_2]&=\frac32 X_1 +\frac12 X_2,    & [Y_1,Y_2]&=-Y_1+Y_2,\\
[X_1,Y_2]&=\frac32 Y_1+X_2,             & [Y_1,X_2]&=-X_1+\frac12 Y_2.
\end{align*}
If we denote by $G$ the Lie group whose Lie algebra is $\mathfrak{g}$, we can define on it a left-invariant paracontact metric structure with
\[
\varphi \xi=0,  \quad \varphi X_i=X_i,  \quad \varphi Y_i=-Y_i, \quad  \eta(\xi)=1,  \quad \eta(X_i)=\eta(Y_i)=0, \quad i=1,2.
\]
and whose only non-vanishing components of the metric are
\[
g(\xi,\xi)=g(X_1,Y_1)=1, \quad g(X_2,Y_2)=-1.
\]
We can check that $h X_i=Y_i$ and $h Y_i=0$, $i=1,2$, so $h^2=0$ and $\text{rank}(h)=2$.

Moreover, long but straightforward computations give us that the manifold is a $(-1,0)$-space.
\end{example}

We show below examples of dimension $7$ with $\text{rank}(h)=2$ and $\text{rank}(h)=3$, respectively. Higher-dimensional examples could be constructed analogously.

\begin{example}[$7$-dimensional paracontact metric $(-1,0)$-space with  $\text{rank}(h)=2$]  \label{ex-dim7-mu0-rank2}
Let us define a $7$-dimensional Lie algebra $\mathfrak{g}$ of basis $\{\xi,X_1,Y_1,X_2,Y_2,X_3,Y_3 \}$ whose only non-vanishing Lie brackets are
\begin{align*}
[\xi,X_1]&=[X_1,X_3]=Y_1+X_1+X_2,           & [\xi,Y_1]&=[Y_1,X_3]=-Y_1+Y_2,\\
[\xi,X_2]&=[X_2,X_3]=Y_2+X_1+X_2,           & [\xi,Y_2]&=[Y_2,X_3]=Y_1-Y_2, \\
[X_1,Y_1]&=2\xi+\sqrt2(X_2+ Y_2),           & [X_2,Y_2]&=-2\xi+\sqrt2X_1, \\
[X_3,Y_3]&=-2\xi-2 X_3,                     & [X_1,X_2]&=-[Y_1,X_2]=\sqrt{2}X_1,\\
[X_1,Y_2]&=\sqrt2( Y_1+X_2),                & [Y_1,Y_2]&=\sqrt2(-Y_1+Y_2).
\end{align*}
If we denote by $G$ the Lie group whose Lie algebra is $\mathfrak{g}$, we can define on it a left-invariant paracontact metric structure with
\[
\varphi \xi=0,  \quad \varphi X_i=X_i,  \quad \varphi Y_i=-Y_i, \quad  \eta(\xi)=1,  \quad \eta(X_i)=\eta(Y_i)=0, \quad i=1,2,3.
\]
and whose only non-vanishing components of the metric are
\[
g(\xi,\xi)=g(X_1,Y_1)=1, \quad g(X_2,Y_2)=g(X_3,Y_3)=-1.
\]
A direct computation gives that $h X_i=Y_i$, $i=1,2$, $hX_3=0$ and $h Y_i=0$, $i=1,2,3$, so $h^2=0$ and $\text{rank}(h)=2$.

Moreover, long but straightforward computations give us that the manifold is a $(-1,0)$-space.
\end{example}

\begin{example}[$7$-dimensional paracontact metric $(-1,0)$-space with  $\text{rank}(h)=3$]  \label{ex-dim7-mu0-rank3}
Let us define a $7$-dimensional Lie algebra $\mathfrak{g}$ of basis $\{\xi,X_1,Y_1,X_2,Y_2,X_3,Y_3 \}$ whose Lie brackets are
\begin{align*}
[\xi,X_1]&=Y_1+X_1+X_2,                     & [\xi,Y_1]&=-Y_1+Y_2,\\
[\xi,X_2]&=Y_2+X_1+X_2,                     & [\xi,Y_2]&=Y_1-Y_2, \\
[\xi,X_3]&=X_3+Y_3,                         & [\xi,Y_3]&=-Y_3, \\
[X_1,Y_1]&=2\xi+\sqrt2(X_2+ Y_2),           & [X_2,Y_2]&=-2\xi+\sqrt2X_1, \\
[X_3,Y_3]&=-2\xi+\sqrt2 (X_1-X_2-Y_2),      & [X_1,X_2]&=-[Y_1,X_2]=\sqrt{2}X_1,\\
[X_1,Y_2]&=\sqrt2( Y_1+X_2),                & [Y_1,Y_2]&=\sqrt2(-Y_1+Y_2),\\
[X_1,Y_3]&=-[X_2,X_3]=[X_2,Y_3]=\sqrt2 X_3, & [Y_1,Y_3]&=-[Y_2,X_3]=[Y_2,Y_3]=\sqrt2 Y_3.\\
\end{align*}
If we denote by $G$ the Lie group whose Lie algebra is $\mathfrak{g}$, we can define on it a left-invariant paracontact metric structure with
\[
\varphi \xi=0,  \quad \varphi X_i=X_i,  \quad \varphi Y_i=-Y_i, \quad  \eta(\xi)=1,  \quad \eta(X_i)=\eta(Y_i)=0, \quad i=1,2,3.
\]
and whose only non-vanishing components of the metric are
\[
g(\xi,\xi)=g(X_1,Y_1)=1, \quad g(X_2,Y_2)=g(X_3,Y_3)=-1.
\]
Therefore, $h X_i=Y_i$, $i=1,2,3$ and $h Y_i=0$, $i=1,2,3$, thus $h^2=0$ and $\text{rank}(h)=3$. Moreover, direct computations give us that the manifold is a $(-1,0)$-space.
\end{example}

\begin{remark}
It is worth mentioning that Theorem~\ref{th-h} is true only pointwise, i.e. $\text{rank}(h_p)$ does not need to be the same for every $p\in M$. However, there are no known examples of paracontact metric $(-1,\mu)$-spaces where $\text{rank}(h)$ is not constant.
\end{remark}

\section{Another manifold with $h^2=0$ but $h\neq0$}

Lastly, we would like to remark that, although the first examples of $h^2=0$ but $h\neq0$ appeared (quite naturally) in the context of paracontact metric $(\kappa,\mu)$-spaces, it is not difficult to construct paracontact metric manifolds satisfying these properties which are not $(-1,\mu)$-spaces.

\begin{example}
Let us take the  $5$-dimensional Lie algebra $\mathfrak{g}$ of basis $\{\xi,X_1,Y_1,X_2,Y_2 \}$ such that the only non-vanishing Lie brackets are
\[
[\xi,X_1]=Y_1, \quad [X_1,Y_1]=2\xi, \quad [X_2,Y_2]=2(\xi+X_2), \quad [X_1,X_2]=Y_1.
\]
If we denote by $G$ the Lie group whose Lie algebra is $\mathfrak{g}$, we can define on it a left-invariant paracontact metric structure:
\[
\varphi \xi=0,  \quad \varphi X_i=X_i,  \quad \varphi Y_i=-Y_i, \quad  \eta(\xi)=1,  \quad \eta(X_i)=\eta(Y_i)=0, \quad i=1,2.
\]
The only non-vanishing components of the metric are
\[
g(\xi,\xi)=g(X_1,Y_1)=g(X_2,Y_2)=1.
\]
Therefore, $h X_1=Y_1$ and $h Y_1=hX_2=hY_2=0$, so $h^2=0$ but $h\neq0$.

Although $R(X,\xi)\xi=-X+2hX$ for all vector field $X$ orthogonal to $\xi$, we can check that $R(X_1,X_2)\xi =-2Y_1\neq0$, so the manifold is not a $(-1,2)$-space.
\end{example}

\begin{remark}
Note that the previous Lie algebra coincides with the one of Example \ref{ex-mu2-h1} for $n=2$ (hence  the form of $h$ coincides) but that the construction of the paracontact metric structure is not the same, since both $\varphi$ and $g$ are defined differently.
\end{remark}

\textbf{Acknowledgements.}
The author would like to thank Prof. Mart\'in Avenda\~no for his valuable suggestions and insights.


\end{document}